\documentclass[12pt, cd]{amsart} 
\usepackage{graphicx,amssymb}
\usepackage{mathrsfs}
\usepackage{amssymb,amsmath,amsthm,color}
\usepackage{graphicx}
\usepackage{hyperref}
\usepackage{url} 
\usepackage{setspace}
\usepackage[margin=1in]{geometry}

\newtheorem{thm}{Theorem}[section]

\newtheorem{rem}[thm]{Remark}
\numberwithin{equation}{section}

\newcommand{\R}{\mathbb R}

\def\TagOnRight

\def\R {\mathbb{R}}

\newcommand{\be}{\begin{equation}}
\newcommand{\ee}{\end{equation}}
\newcommand{\bea}{\begin{eqnarray}}
\newcommand{\eea}{\end{eqnarray}}
\newcommand{\Bea}{\begin{eqnarray*}}
\newcommand{\Eea}{\end{eqnarray*}}
\newcommand{\bt}{\begin{Theorem}}
\newcommand{\et}{\end{Theorem}}

\newcommand{\bpr}{\begin{Proposition}}
\newcommand{\epr}{\end{Proposition}}

\newcommand{\bl}{\begin{Lemma}}
\newcommand{\el}{\end{Lemma}}
\newcommand{\bi}{\begin{itemize}}
\newcommand{\ei}{\end{itemize}}

\newtheorem{Definition}{Definition}[section]
\newtheorem{Theorem}[Definition]{Theorem}
\newtheorem{Lemma}[Definition]{Lemma}
\newtheorem{Proposition}[Definition]{Proposition}

\usepackage{mathrsfs}
\usepackage{amsthm}
\usepackage{hyperref}
\usepackage{comment}

\begin{document}
\baselineskip16pt

\title[Global$-$wellposedness for FHTE]{Global well-posedness for  fractional  Hartree  equation on modulation spaces and Fourier algebra}
\author{Divyang G. Bhimani}
\address {Department of Mathematics, University of Maryland, College Park, MD 20742}
 \email {dbhimani@math.umd.edu}

\maketitle
\begin{abstract} We study the Cauchy problem for  fractional  Schr\"odinger equation with cubic convolution nonlinearity ($i\partial_t u - (-\Delta)^{\frac{\alpha}{2}}u\pm (K\ast |u|^2) u =0$) with Cauchy data in the modulation spaces $M^{p,q}(\mathbb R^{d}).$ For $K(x)= |x|^{-\gamma}$ $ (0< \gamma< \text{min} \{\alpha, d/2\})$, we establish global well-posedness results in $M^{p,q}(\mathbb R^{d}) (1\leq p \leq 2, 1\leq q < 2d/ (d+\gamma))$   when $\alpha =2, d\geq 1$,  and with radial Cauchy data when  $d\geq 2, \frac{2d}{2d-1}< \alpha < 2.
$ Similar results are proven in  Fourier algebra $\mathcal{F}L^1(\mathbb R^d) \cap L^2(\mathbb R^d).$
\end{abstract}
\section{Introduction}
We study fractional nonlinear Schr\"odinger equation with cubic convolution non linearity
\begin{eqnarray}\label{fHTE}
i\partial_t u - (-\Delta)^{\frac{\alpha}{2}}u= (K\ast |u|^2) u, u(x,0)= u_0(x)   
\end{eqnarray}
where $u:\mathbb R_t \times \mathbb R_x^d \to \mathbb C, u_0:\mathbb R^d \to \mathbb C,$ $K$ 
denotes the Hartree kernel
\begin{eqnarray}\label{hk}
\   \ \ \ \ \ \ \ \ \ \ \ \  \ \ \ \   \   \  \   \ K(x)= \frac{\lambda}{|x|^{\gamma}}, (\lambda \in \mathbb R, \gamma >0, x\in \mathbb R^{d}),
\end{eqnarray}
 and $\ast$ denotes the convolution in $\mathbb R^d.$ The fractional Laplacian is defined as 
\begin{eqnarray}
\mathcal{F}[(-\Delta)^{\alpha/2}u] (\xi) = |\xi|^{\alpha} \mathcal{F}u (\xi)
\end{eqnarray}
where $\mathcal{F}$ denotes the Fourier transform.  The equation \eqref{fHTE}  is known as the Hartree type equation with fractional Laplacian, and  we call it fractional Hartree equation (FHTE).

In recent years, there has been a great deal of interest \cite{cabre, laskin} in using fractional Laplacians to model physical phenomena.  It was formulated by N. Laskin \cite{laskin} as a result of extending the Feynman path integral from the Brownian-like to L\'evy-like quantum mechanical paths.  Specifically,  when  $\alpha = 1$, \eqref{fHTE} can be used to describe the dynamics of pseudo-relativistic boson stars in the mean-field limit, and  when $\alpha =2$ the L\'evy motion becomes Brownian motion. Thus,  in  tbe last couple of years many authors have studied the  Cauchy problem  for  fractional Hartree equation with Cauchy data in Sobolev spaces, see  for instance \cite{4a, gz} and the reference therein. 

On the other hand,  there has been a lot of ongoing interest to study Cauchy problem  for nonlinear dispersive  equations with the Cauchy data in modulation spaces  $M^{p,q}(\mathbb R^d)$ (See Definition \ref{ms} below).  Generally speaking, the Cauchy data in a modulation spaces are rougher than  any given one in a  fractional Bessel potential space and this low regularity is desirable in many situations. It is well known that  \cite{chen, benyi, bao} Schr\"odinger  propogator $e^{-it (-\Delta)^{\alpha/2}}$ is bounded on $M^{p,q}(\mathbb R^d)$ but most authors have studied  local \cite{ambenyi} and global \cite{wang2}  well-posedness (for small data)  for a  nonlinear dispersive  equations with Cauchy data in $M^{p,1}(\mathbb R^d).$   As one advantage to consider Cauchy data in $M^{p,1}(\mathbb R^d)$ is the fact that it is an algebra under point wise multiplication. There has not been  much progress to consider the Cauchy data in the  modulation spaces $M^{p,q}   (1<q\leq \infty).$  However, there has been an attempt to get  global well-posedness  results for a large data in modulation spaces but yet we do not know the complete answer to this, see for instance, open question raised by Baoxiang-Ruzhansky-Sugimoto \cite[p.280]{rsw}.\\

Taking these considerations into  account, in this article, first  we study  global well-posedness result for  fractional Hartree  equation \eqref{fHTE}  with Cauchy data in modulation spaces $M^{p,q}(\mathbb R^d)$ (without any smallness assumption  for the Cauchy data).   We could include family of Banach spaces which  may not be algebra under point wise multiplication, say $M^{p,q}(\mathbb R^d) (q\neq 1).$  We denote the Banach space $X$ of radial functions by $X_{rad}.$  Specifically, we  obtain following results.

\begin{Theorem}\label{MT}
Let $K$ be given by \eqref{hk} with  $0 < \gamma < \min\{\alpha, d/2\},$ and  $ 1\leq p \leq  2, 1\leq q < \frac{2d}{d+ \gamma}.$
\begin{enumerate}
\item \label{MTF} Assume that $u_{0}\in M_{rad}^{p,q}(\mathbb R^{d}), $  $d\geq 2, \frac{2d}{2d-1}< \alpha < 2.$  Then
 there exists a  unique global solution of \eqref{fHTE} such that $$u\in C(\mathbb R, M_{rad}^{p,q}(\mathbb R^{d})) \cap L^{4\alpha/\gamma}_{loc}(\mathbb R, L^{4d/(2d-\gamma)} (\mathbb R^d)).$$
 \item \label{MTS} Assume that $u_{0}\in M^{p,q}(\mathbb R^{d}), d\in \mathbb N, \alpha =2.$ Then
 there exists a  unique global solution of \eqref{fHTE} such that $$u\in C(\mathbb R, M^{p,q}(\mathbb R^{d})) \cap L^{8/\gamma}_{loc}(\mathbb R, L^{4d/(2d-\gamma)} (\mathbb R^d)) .$$
\end{enumerate}
\end{Theorem}
We note that   $M^{p,q}(\mathbb R^d)  \subset L^p(\mathbb R^d) (q \leq \min \{p, p'\})$ is sharp embedding and  up to now we cannot get the  global well-posedness of   fractional Hartree equation  in $L^p(\mathbb R^d) (1\leq p < 2)$ but in $M^{p,q}(\mathbb R^d)$ (see Theorem \ref{MT}).  We also note that recently  the author \cite[Theorem 1.1]{dgb} have proved the global well-posedness result  for the classical Hartree equation (that is  \eqref{fHTE} with $\alpha =2$)  with Cauchy date  in $M^{1,1}(\mathbb R^d),$ and this result  has been extended by  Manna \cite[Theorem 1.1]{rm} with Cauchy data in $M^{p,p}(\mathbb R^d)$ for $1\leq p < \frac{2d}{d+\gamma}$. Thus,  our Theorem \ref{MT}\eqref{MTS} extends the existing results in the literature 
so far.  It is also worth noting following
\begin{rem} The  nonlinear Schr\"odinger equation (NLS) 
\begin{eqnarray*}
i\partial_{t}u+ \Delta u = |u|^2 u,
\end{eqnarray*}
 is locally \cite{ambenyi} well-posed in $M^{p,1}(\mathbb R^d)$ but it is not yet clear whether  it is  globally well-posed or not (see for instance \cite[p.280]{rsw}). Also, it is not clear whether NLS  is locally wellposed in $M^{p,q}(\mathbb R^d) (q\neq 1).$ But in contrast fractional Hartree equation is  globally well-posed in $M^{p,q}(\mathbb R^d)$ (see Theorem \ref{MT}).
\end{rem}

Next we study  fractional Hartree equation   \eqref{fHTE} with  Cauchy data in  Fourier algebra $\mathcal{F}L^1(\mathbb R^d)$ (See Section \ref{np} below) and the space of  square integrable functions. We note that for $s\geq d/2,$ $H^{s}(\mathbb R^d) \hookrightarrow  \mathcal{F}L^{1}(\mathbb R^d) \hookrightarrow M^{\infty, 1}(\mathbb R^d).$  We recall recently  Carles-Mouzaoui \cite{carle} have proved global well-posedness result for classical Hartree  equation  (that is \eqref{fHTE} with $\alpha =2$)  with Cauchy data in  $L^2(\mathbb R^d) \cap \mathcal{F}L^1(\mathbb R^d)$.  We extend this result for the fractional Hartree equation. Specifically, we prove following
\begin{Theorem}\label{jcm}  Let $d\geq 2, \frac{2d}{2d-1}< \alpha < 2, 0 < \gamma < \min\{\alpha, d/2\}.$ Assume that $u_{0}\in L_{rad}^{2}(\mathbb R^{d}) \cap \mathcal{F}L^1(\mathbb R^d)$, and let $K$ be given by \eqref{hk} with $\lambda \in \mathbb R.$  Then
 there exists a  unique global solution of \eqref{fHTE} such that $$u\in C(\mathbb R, L^2_{rad} (\mathbb R^d)) \cap  C(\mathbb R, \mathcal{F}L^1 (\mathbb R^d)) \cap L^{4\alpha/\gamma}_{loc}(\mathbb R, L^{4d/(2d-\gamma)} (\mathbb R^d)).$$
\end{Theorem}

This article is organized as follows. We will give basic notations and tools that will be used  throughout this article, in Section  \ref{np}.  This includes all definitions  of function spaces  and some of their properties, and known lemmas. Section \ref{smt} is devoted to proving  result  concerning global well-posedness result for the Cauchy data in $L^2(\mathbb R^d).$ We shall see this will turn out to be  one of the main tools to obtain global well-posedness results in the realm of modulation spaces and Fourier algebra. In Section \ref{pmr}, we shall prove  our main  Theorems \ref{MT} and \ref{jcm}.
\section{Notations and Preliminaries} \label{np} 
\subsection{Notations} The notation $A \lesssim B $ means $A \leq cB$ for a some constant $c > 0 $, whereas $ A \asymp B $ means $c^{-1}A\leq B\leq cA $ for some $c\geq 1$. The symbol $A_{1}\hookrightarrow A_{2}$ denotes the continuous embedding  of the topological linear space $A_{1}$ into $A_{2}.$ Let $I\subset \mathbb R$ be an interval. Then the norm of the space-time Lebesgue space $L^{p}(I, L^q(\mathbb R^d))$ is defined by
$$\|u\|_{L^p(I, L^q(\mathbb R^d))}=\|u\|_{L^{p}_I L^q_x}= \left(\int_{I} \|u(t)\|^p_{L^q_x} dt \right)^{1/p}.$$
If there is no confusion, we simply write
$$\|u\|_{L^p(I, L^q)}=\|u\|_{L^{p}_I L^q_x}=\|u\|_{L^{p,q}_{t,x}}.$$
The Schwartz space is denoted by  $\mathcal{S}(\mathbb R^{d})$ (with its usual topology), and the space of tempered distributions is  denoted by $\mathcal{S'}(\mathbb R^{d}).$ For $x=(x_1,\cdots, x_d), y=(y_1,\cdots, y_d) \in \mathbb R^d, $ we put $x\cdot y = \sum_{i=1}^{d} x_i y_i.$
Let $\mathcal{F}:\mathcal{S}(\mathbb R^{d})\to \mathcal{S}(\mathbb R^{d})$ be the Fourier transform  defined by  
\begin{eqnarray}
\mathcal{F}f(w)=\widehat{f}(w)=\int_{\mathbb R^{d}} f(t) e^{- 2\pi i t\cdot w}dt, \  w\in \mathbb R^d.
\end{eqnarray}
Then $\mathcal{F}$ is a bijection  and the inverse Fourier transform  is given by
\begin{eqnarray}
\mathcal{F}^{-1}f(x)=f^{\vee}(x)=\int_{\mathbb R^{d}} f(w)\, e^{2\pi i x\cdot w} dw,~~x\in \mathbb R^{d},
\end{eqnarray}
and this Fourier transform can be uniquely extended to $\mathcal{F}:\mathcal{S}'(\mathbb R^d) \to \mathcal{S}'(\mathbb R^d).$ We denote by $\mathcal{F}L^p(\mathbb R^d)$ the Fourier-Lebesgue space, specifically,
$$\mathcal{F}L^p(\mathbb R^d)= \{f\in \mathcal{S}'(\mathbb R^d): \hat{f}\in L^{p}(\mathbb R^d) \}.$$
The  $\mathcal{F}L^{p} (\mathbb R^d)-$norm is denoted by
 $$\|f\|_{\mathcal{F}L^{p}}= \|\hat{f}\|_{L^{p}} \ (f\in \mathcal{F}L^{p}(\mathbb R^{d})).$$
If $s\in \mathbb R,$ we put  $\langle \xi \rangle^{s} = (1+ |\xi|^2)^{s/2} \ (\xi \in \mathbb R^d),$ and define Sobolev space $H^s(\mathbb R^d)$ to be 
$$H^s(\mathbb R^d)=\{f\in \mathcal{S}'(\mathbb R^d): \mathcal{F}^{-1} [\langle \cdot \rangle \mathcal{F}(f)] \in L^2(\mathbb R^d) \}.$$

\subsection{Modulation  spaces}
In 1983, Feichtinger \cite{HG4} introduced a  class of Banach spaces,  the so called modulation spaces,  which allow a measurement of space variable and Fourier transform variable of a function or distribution on $\mathbb R^d$ simultaneously using the short-time Fourier transform(STFT). The  STFT  of a function $f$ with respect to a window function $g \in {\mathcal S}(\R^d)$ is defined by
\begin{eqnarray}\label{stft}
V_{g}f(x,w)= \int_{\mathbb R^{d}} f(t) \overline{g(t-x)} e^{-2\pi i w\cdot t}dt,  \  (x, w) \in \mathbb R^{2d}
\end{eqnarray}
 whenever the integral exists.
For $x, w \in \R^d$ the translation operator $T_x$ and the modulation operator $M_w$ are
defined by $T_{x}f(t)= f(t-x)$ and $M_{w}f(t)= e^{2 \pi i w\cdot t} f(t).$ In terms of these
operators the STFT may be expressed as
\begin{eqnarray}
\label{ipform} V_{g}f(x,y)=\langle f, M_{y}T_{x}g\rangle = e^{-2\pi i x\cdot y} (f\ast M_yg^*)(x)
\end{eqnarray}
 where $g^*(y)= \overline{g(-y)}$  and $\langle f, g\rangle$ denotes the inner product for $L^2$ functions,
or the action of the tempered distribution $f$ on the Schwartz class function $g$.  Thus $V: (f,g) \to V_g(f)$ extends to a bilinear form on $\mathcal{S}'(\R^d) \times \mathcal{S}(\R^d)$ and $V_g(f)$ defines a uniformly continuous function on $\R^{d} \times \R^d$ whenever $f \in \mathcal{S}'(\R^d) $ and $g \in  \mathcal{S}(\R^d)$.
\begin{Definition}[modulation spaces]\label{ms} Let $1 \leq p,q \leq \infty, s\in \mathbb R$ and $0\neq g \in{\mathcal S}(\R^d)$. The  weighted  modulation space   $M_s^{p,q}(\R^d)$
is defined to be the space of all tempered distributions $f$ for which the following  norm is finite:
$$ \|f\|_{M_s^{p,q}}=  \left(\int_{\R^d}\left(\int_{\R^d} |V_{g}f(x,y)|^{p} dx\right)^{q/p} \langle y \rangle^{sq}  \, dy\right)^{1/q},$$ for $ 1 \leq p,q <\infty$. If $p$ or $q$ is infinite, $\|f\|_{M_s^{p,q}}$ is defined by replacing the corresponding integral by the essential supremum. 
\end{Definition}

\begin{rem}
\label{equidm}
The definition of the modulation space given above, is independent of the choice of 
the particular window function.  See \cite[Proposition 11.3.2(c), p.233]{gro}.  If $s=0,$ we simply write $M^{p,q}_s= M^{p,q}.$
\end{rem}

In the next lemma we  collect basic properties of modulation spaces which we shall need later.
 
\begin{Lemma}\label{rl} Let $p,q, p_{i}, q_{i}\in [1, \infty]$  $(i=1,2), s_1, s_2 \in \mathbb R.$ Then
\begin{enumerate}
\item \label{ir} $M_{s_1}^{p_{1}, q_{1}}(\mathbb R^{d}) \hookrightarrow M_{s_2}^{p_{2}, q_{2}}(\mathbb R^{d})$ whenever $p_{1}\leq p_{2},q_{1}\leq q_{2},$ and $ s_2 \leq s_1.$

\item \label{el} $M^{p,q_{1}}(\mathbb R^{d}) \hookrightarrow L^{p}(\mathbb R^{d}) \hookrightarrow M^{p,q_{2}}(\mathbb R^{d})$ holds for $q_{1}\leq \text{min} \{p, p'\}$ and $q_{2}\geq \text{max} \{p, p'\}$ with $\frac{1}{p}+\frac{1}{p'}=1.$
\item \label{rcs} $M^{\min\{p', 2\}, p}(\mathbb R^d) \hookrightarrow \mathcal{F} L^{p}(\mathbb R^d)\hookrightarrow M^{\max \{p',2\},p}(\mathbb R^d),  \frac{1}{p}+\frac{1}{p'}=1.$
\item $\mathcal{S}(\mathbb R^{d})$ is dense in  $M^{p,q}(\mathbb R^{d})$ if $p$ and $q<\infty.$
\item \label{fi} The Fourier transform $\mathcal{F}:M^{p,p}(\mathbb R^{d})\to M^{p,p}(\mathbb R^{d})$ is an isomorphism.
\item The space  $M^{p,q}(\mathbb R^{d})$ is a  Banach space.
\item \label{ic}The space $M^{p,q}(\mathbb R^{d})$ is invariant under complex conjugation.
\end{enumerate}
\end{Lemma}
\begin{proof}
All these statements are well-known and the interested reader may find a proof in \cite{gro, HG4, baob}.  For the proof of statements  \eqref{ir}, \eqref{el} and \eqref{rcs}, see  \cite[Theorem 12.2.2]{gro}, \cite{nonom},  and \cite[Corollary 1.1]{cks} respectively.  The proof  for the statement (5)  can be derived from  the fundamental identity of time-frequency analysis: 
$$ V_gf(x, w) = e^{-2 \pi i x \cdot w } \, V_{\widehat{g}} \widehat{f}(w, -x),$$
which is easy to obtain.The proof of the statement (7) is trivial, indeed, we have $\|f\|_{M^{p,q}}=\|\bar{f}\|_{M^{p,q}}.$ 
\end{proof}
\begin{Theorem} \label{pl} Let $p,q, p_{i}, q_{i}\in [1, \infty]$  $(i=0,1,2).$
If   $\frac{1}{p_1}+ \frac{1}{p_2}= \frac{1}{p_0}$ and $\frac{1}{q_1}+\frac{1}{q_2}=1+\frac{1}{q_0}, $ then
\begin{eqnarray}\label{prm}
M^{p_1, q_1}(\mathbb R^{d}) \cdot M^{p_{2}, q_{2}}(\mathbb R^{d}) \hookrightarrow M^{p_0, q_0}(\mathbb R^{d});
\end{eqnarray}
with norm inequality $\|f g\|_{M^{p_0, q_0}}\lesssim \|f\|_{M^{p_1, q_1}} \|g\|_{M^{p_2,q_2}}.$
In particular, the  space $M^{p,q}(\mathbb R^{d})$ is a poinwise $\mathcal{F}L^{1}(\mathbb R^{d})$-module, that is, it satisfies
\begin{eqnarray}\label{mp}
\|fg\|_{M^{p,q}} \lesssim \|f\|_{\mathcal{F}L^{1}} \|g\|_{M^{p,q}}.
\end{eqnarray} 
\end{Theorem}
\begin{proof}
The product relation \eqref{prm} between modulation spaces is well known and we refer the interested reader to \cite{ambenyi} and  since $\mathcal{F}L^{1}(\mathbb R^{d}) \hookrightarrow M^{\infty, 1}(\mathbb R^{d})$, the desired inequality \eqref{mp} follows.
\end{proof}
For $f\in \mathcal{S}(\mathbb R^{d}),$ we define the fractional Schr\"odinger propagator $e^{it(-\Delta)^{\alpha/2}}$ for $t, \alpha \in \mathbb R$ as follows:
\begin{eqnarray*}
\label{sg}
U(t)f(x)=e^{it (-\Delta)^{\alpha/2}}f(x)= \int_{\mathbb R^d}  e^{i \pi t|\xi|^{\alpha}}\, \hat{f}(\xi) \, e^{2\pi i \xi \cdot x} \, d\xi.
\end{eqnarray*}

The next proposition shows that the uniform boundedness of the Schr\"odinger propagator $e^{it(-\Delta)^{\alpha/2}}$ on modulation spaces. In fact, using  \cite[Theorems 1, 2]{chen} and Lemma \ref{rl}\eqref{ir}, we have 
\begin{Proposition}[\cite{chen}] \label{uf} Let $\frac{1}{2} < \alpha \leq 2, 1 \leq p,q \leq \infty.$ Then $$ \|U(t)f \|_{M^{p,q}}\leq  (1+t)^{d\left| \frac{1}{p}-\frac{1}{2} \right|} \|f\|_{M^{p,q}}.$$
\end{Proposition}

Finally, we note that there is also an equivalent  definition of modulation spaces using frequency-uniform decomposition techniques (which is quite similar in the spirit of Besov spaces), independently  studied by Wang et al. in \cite{bao}, which has turned out to be very fruitful in PDEs.  For a brief survey of modulation spaces and nonlinear evolution equations,  we refer the interested reader to \cite{rsw} and  for further  reading from the PDEs viewpoint we refer to \cite{baob} and the references therein.

\section{Global wellposedness in $L^2(\mathbb R^d)$}\label{smt}
In this section we prove global well-posedness for fractional Hartree equation with Cauchy date in  $L^2(\mathbb R^d).$  We shall see this  result will turn out to be one of the key steps in establishing the global well-posedness results in modulation spaces and Fourier algebra.

\begin{Definition} A pair $(q,r)$ is $\alpha-$fractional admissible if  $q\geq 2, r\geq 2$ and
$$\frac{\alpha}{q} =  d \left( \frac{1}{2} - \frac{1}{r} \right).$$ 
\end{Definition}

\begin{Proposition}$($\cite[Corollary 3.10]{guo}$)$ \label{fst} Let  $d\ge 2$ and  $\frac{2d}{2d-1} < \alpha \leq 2.$  Assume that $u_0$ and $F$ are radial. Then 
\begin{enumerate}
\item For any $\alpha-$fractional admissible pair  $(p,q),$ there exists $C_q$ such that 
$$\|U(t) \phi  \|_{L^{p,q}_{t,x}} \leq C_q \|\phi \|_{L^2}, \   \forall \phi \in L^2 (\mathbb R^d).$$
\item Define 
$$DF(t,x) = \int_0^t U(t-\tau )F(\tau,x) d\tau.$$
For all $\alpha-$fractional  admissible pair  $(p_i,q_i), (i=1, 2)$, there exists  $C$ (constant) such that for all intervals $I \ni 0, $
$$ \|D(F)\|_{L^{p_1, q_1}_{t,x}}  \leq C  \|F\|_{L^{p_2', q_2'}_{t,x}}, \ \forall F \in L^{p_2'} (I, L^{q_2'})$$  where $p_i'$ and $ q_i'$ are H\"older conjugates of $p_i$ and $q_i'$
respectively.
\end{enumerate}
\end{Proposition}
\begin{rem}\label{jr} If $\alpha =2,$ we do not need radial assumption on initial data $u_0$ and on nonlinearity $F$  in the Strichartz estimate. Specifically,  if $\alpha =2,$ Proposition \ref{fst} is true for any $d\geq 1, u_0$ and $F$. See \cite{kt}.
\end{rem}
We also need to work with the  convolution with the Hartree potential  $|x|^{-\gamma},$  so for the convenience of reader we recall:
\begin{Proposition}[Hardy-Littlewood-Sobolev inequality] \label{hls}Assume that  $0<\gamma< d$ and $1<p<q< \infty$ with
$$\frac{1}{p}+\frac{\gamma}{d}-1= \frac{1}{q}.$$
Then the map $f \mapsto |x|^{-\gamma}\ast f$ is bounded from $L^p(\mathbb R^d)$ to $L^q(\mathbb R^d):$
$$\|x|^{-\gamma}\ast f\|_{L^q} \leq C_{d,\gamma, p} \|f\|_{L^p}.$$
\end{Proposition}

Now in next proposition  we shall prove global wellposedness in $L^2(\mathbb R^d)$ for the fractional Hartree equation \eqref{fHTE}.
\begin{Proposition}\label{miD}
Let $d\geq 2,$  $\frac{2d}{2d-1} < \alpha \leq 2,$  and  $K$ be given by \eqref{hk} with $\lambda \in \mathbb R$ and $0<\gamma < \text{min} \{\alpha, d\}.$ If $u_{0}\in L_{rad}^{2}(\mathbb R^{d}),$ then \eqref{fHTE} has a unique global solution 
$$u\in C(\mathbb R, L_{rad}^{2}(\mathbb R^d))\cap L^{4\alpha/\gamma}_{loc}(\mathbb R, L^{4d/(2d-\gamma)} (\mathbb R^d)).$$ 
In addition, its $L^{2}-$norm is conserved, 
$$\|u(t)\|_{L^{2}}=\|u_{0}\|_{L^{2}}, \   \forall t \in \mathbb R,$$
and for all $\alpha-$ fractional admissible pairs  $(p,q), u \in L_{loc}^{p}(\mathbb R, L^{q}(\mathbb R^d)).$
\end{Proposition}
\begin{proof} By Duhamel's formula, we write \eqref{fHTE}
as 
$$u(t)=U(t)u_0- i \int_0^t U(t-\tau) (K \ast |u|^2)u(\tau) d\tau:= \Phi(u)(t).$$
Put $ s= \frac{\alpha}{2}$.  We introduce the space
\begin{eqnarray*}
Y(T)  & =  &\{ \phi \in C\left([0,T], L_{rad}^2(\mathbb R^d) \right): \|\phi \|_{L^{\infty}([0, T], L^2)}  \leq 2 \|u_0\|_{L^2}, \\
&&  \|\phi\|_{L^{\frac{8s}{\gamma}} ([0,T], L^{\frac{4d}{2d-\gamma}}  ) } \lesssim \|u_0\|_{L^2}\}
\end{eqnarray*}
and the distance 
$$d(\phi_1, \phi_2)= \|\phi_1 - \phi_2 \|_{L^{\frac{8s}{\gamma} }\left( [0, T], L^{\frac{4d}{(2d- \gamma)}}\right)}.$$ Then $(Y, d)$ is a complete metric space. Now we show that $\Phi$ takes $Y(T)$ to $Y(T)$ for some $T>0.$
We put 
$$ \ q= \frac{8s}{\gamma}, \  r= \frac{4d}{2d- \gamma}.$$ 
Note that $(q,r)$ is $\alpha-$fractional admissible and 
$$ \frac{1}{q'}= \frac{4s- \gamma}{4s} + \frac{1}{q}, \  \frac{1}{r'}= \frac{\gamma}{2d} + \frac{1}{r}.$$
Let $(\bar{q}, \bar{r}) \in \{ (q,r), (\infty, 2) \}.$ By Proposition \ref{fst} and  H\"older inequality, we have 
\begin{eqnarray*}
\|\Phi(u)\|_{L_{t,x}^{\bar{q}, \bar{r}}} &  \lesssim &  \|u_0\|_{L^2} + \|(K \ast |u|^2)u \|_{L_{t,x}^{q',r'}}\\
& \lesssim &  \|u_0\|_{L^2} + \|K \ast |u|^2\|_{L_{t,x}^{\frac{4s}{4s-\gamma}, \frac{2d}{\gamma}}} 
\|u\|_{L^{q,r}_{t,x}}.
\end{eqnarray*}
Since $0<\gamma< \min \{\alpha, d \}$, by Proposition \ref{hls},  we  have 
\begin{eqnarray*}
\|K \ast |u|^2\|_{L_{t,x}^{\frac{4s}{4s-\gamma}, \frac{2d}{\gamma}}}  & = &  \left\|  \|K\ast |u|^2\|_{L_x^{\frac{2d}{\gamma}}} \right\|_{L_t^{\frac{4s}{4s- \gamma}}}\\
& \lesssim &   \left \| \||u|^2\|_{L_x^{\frac{2d}{2d- \gamma}}} \right\|_{L_t^{\frac{4s}{4s- \gamma}}} \\
& \lesssim & \|u\|^2_{L_{t,x}^{\frac{8s}{4s- \gamma},r}}\\
& \lesssim & T^{1- \frac{\gamma}{2s}} \|u\|^2_{L_{t,x}^{q,r}}.
\end{eqnarray*}
(In the last inequality we have used inclusion relation for the $L^p$ spaces on finite measure spaces: $\|\cdot\|_{L^p(X)} \leq \mu(X)^{\frac{1}{p}-\frac{1}{q}} \|\cdot \|_{L^{q}(X)}$ if measure of $X=[0,T]$ is finite, and $ 0<p<q<\infty$.) 
Thus, we  have 
$$
\|\Phi(u)\|_{L_{t,x}^{\bar{q}, \bar{r}}}   \lesssim   \|u_0\|_{L^2}+T^{1- \frac{\gamma}{2s}} \|u\|^3_{L_{t,x}^{q,r}}.$$
This shows that $\Phi$ maps $Y(T)$ to $Y(T).$  Next, we show $\Phi$
is a  contraction. For this, as 
 calculations performed  before, first we note  that 
\begin{eqnarray}\label{mi}
 \|(K\ast |v|^{2})(v-w)\|_{L_{t,x}^{q',r'}} \lesssim  T^{1-\frac{\gamma}{2s}} \|v\|^2_{L_{t,x}^{q,r}} \|v-w\|_{L_{t,x}^{q,r}}.
\end{eqnarray}
Put $\delta = \frac{8s}{4s-\gamma}.$  Notice that $\frac{1}{q'}= \frac{1}{2}+ \frac{1}{\delta}, \frac{1}{2}= \frac{1}{\delta} + \frac{1}{q},$ and thus by H\"older inequality, we obtain
\begin{eqnarray}\label{mi1}
 \|(K \ast (|v|^{2}- |w|^{2}))w\|_{L_{t,x}^{q',r'}} & \lesssim & \| K\ast \left( |v|^2-|w|^2\right)\|_{L_{t,x}^{2, \frac{2d}{\gamma}}} \|w\|_{L_{t,x}^{\delta, r}} \nonumber \\
 & \lesssim & ( \|K \ast (v (\bar{v}- \bar{w})) \|_{L_{t,x}^{2, \frac{2d}{\gamma}}} \nonumber
 \\
 && + \|K \ast \bar{w}(v-w)) \|_{L_{t,x}^{2, \frac{2d}{\gamma}}} ) \|w\|_{L_{t,x}^{\delta,r}} \nonumber\\
 & \lesssim & \left( \|v\|_{L_{t,x}^{\delta,r}} \|w\|_ {L_{t,x}^{\delta,r}} +\|w\|^2_ {L_{t,x}^{\delta,r}}\right) \|v-w\|_{L_{t,x}^{q,r}}\nonumber\\
 & \lesssim & T^{1-\frac{\gamma}{2s}}  \left( \|v\|_{L_{t,x}^{q,r}} \|w\|_ {L_{t,x}^{q,r}} +\|w\|^2_ {L_{t,x}^{q,r}}\right)  \|v-w\|_{L_{t,x}^{q,r}}
\end{eqnarray}
In view of the identity
$$(K\ast |v|^{2})v- (K\ast |w|^{2})w= (K\ast |v|^{2})(v-w) + (K \ast (|v|^{2}- |w|^{2}))w, $$  \eqref{mi}, and \eqref{mi1} gives
\begin{eqnarray*}
\|\Phi(v)- \Phi(w)\|_{L_{t,x}^{q,r}} & \lesssim &  \|(K\ast |v|^{2})(v-w)\|_{L_{t,x}^{q',r'}} + \|(K \ast (|v|^{2}- |w|^{2}))w\|_{L_{t,x}^{q',r'}}\\
& \lesssim &   T^{1-\frac{\gamma}{2s}}  \left( \|v\|^2_{L_{t,x}^{q,r}}  +\|v\|_{L_{t,x}^{q,r}} \|w\|_ {L_{t,x}^{q,r}} +\|w\|^2_ {L_{t,x}^{q,r}}\right)\\
&&  \|v-w\|_{L_{t,x}^{q,r}}.
\end{eqnarray*}
Thus $\Phi$ is a contraction form $Y(T)$ to $Y(T)$  provided that $T$ is sufficiently small. Then there exists a unique $u \in Y(T)$ solving \eqref{fHTE}. The global  existence of the solution \eqref{fHTE} follows from the conservation of the $L^2-$norm of $u.$ The last property of the proposition then follows from the Strichartz estimates applied with an arbitrary $\alpha-$fractional admissible pair on the left hand side and the same pairs as above on the right hand side.
\end{proof}
In view of Remark \ref{jr}, and  exploiting the ideas from the previous proposition, we obtain 
 
\begin{Proposition}[\cite{carle}]\label{carlep}
Let $d\geq 1,$  and  $K$ be given by \eqref{hk} with $\lambda \in \mathbb R$ and $0<\gamma < \text{min} \{2, d\}.$ If $u_{0}\in L^{2}(\mathbb R^{d}),$ then \eqref{fHTE} has a unique global solution 
$$u\in C(\mathbb R, L^{2}(\mathbb R^d))\cap L^{8/\gamma}_{loc}(\mathbb R, L^{4d/(2d-\gamma)}(\mathbb R^d)).$$ 
In addition, its $L^{2}-$norm is conserved, 
$$\|u(t)\|_{L^{2}}=\|u_{0}\|_{L^{2}}, \   \forall t \in \mathbb R,$$
and for all 2$-$fractional  admissible pairs  $(p,q), u \in L_{loc}^{p}(\mathbb R, L^{q}(\mathbb R^d)).$
\end{Proposition}

\section{Proofs of main results}\label{pmr} 
\subsection{Global well-posedness  in $M^{p,q}(\mathbb R^d)$}
In this subsection, we shall prove Theorem \ref{MT}.
For  convenience of  the reader, we recall
\begin{Proposition}[\cite{baho}]\label{fc} 
Let $d\geq 1,$  $0<\gamma <d$ and $\lambda \in \mathbb R.$ There exists $C=C(d,\gamma)$ such that the Fourier transform of $K$ defined by \eqref{hk} is
\begin{eqnarray}
\widehat{K}(\xi)= \frac{\lambda C}{|\xi|^{d-\gamma}}.
\end{eqnarray}
\end{Proposition}
We start with decomposing Fourier transform of Hartree potential into Lebesgue spaces: indeed, in view of Proposition \ref{fc}, we  have
\begin{eqnarray}\label{dc}
\widehat{K}=k_1+k_2 \in L^{p}(\mathbb R^{d})+ L^{q}(\mathbb R^{d}),
\end{eqnarray}
where  $k_{1}:= \chi_{\{|\xi|\leq 1\}}\widehat{K} \in L^{p}(\mathbb R^{d})$ for all $p\in [1, \frac{d}{d-\gamma})$ and $k_{2}:= \chi_{\{|\xi|>1\}} \widehat{K} \in L^{q}(\mathbb R^{d})$ for all $q\in (\frac{d}{d-\gamma}, \infty].$
\begin{Lemma}\label{iml} Let  $K$ be given by \eqref{hk} with $\lambda \in \mathbb R,$ and $ 0<\gamma < d$, and  $1\leq p \leq 2, 1\leq q < \frac{2d}{d+\gamma}$. Then for any $f\in M^{p,q}(\mathbb R^{d}),$ we have
\begin{eqnarray}\label{d1}
\|(K\ast |f|^{2}) f\|_{M^{p,q}} \lesssim\|f\|_{M^{p,q}}^{3}.
\end{eqnarray}
\end{Lemma}
\begin{proof}
By \eqref{mp}  and  \eqref{dc}, we have
\begin{eqnarray}
\|(K\ast |f|^{2}) f)\|_{M^{p,q}} & \lesssim  & \|K\ast |f|^{2}\|_{\mathcal{F}L^{1}} \|f\|_{M^{p,q}}\nonumber \\
& \lesssim & \left(\|k_{1} \widehat{|f|^{2}}\|_{L^{1}} + \|k_{2} \widehat{|f|^{2}}\|_{L^{1}} \right) \|f\|_{M^{p,q}}. \label{md1}
\end{eqnarray}
By \eqref{dc} and Lemma \ref{pl}\eqref{el}, we have
\begin{eqnarray}
\|k_1 \widehat{|f|^2}\|_{L^1}  & \lesssim &  \|k_1\|_{L^1} \|\widehat{|f|^2}\|_{L^{\infty}} \nonumber \\
& \lesssim & \||f|^2\|_{L^1}= \|f\|_{L^2}^2 \nonumber \\
& \lesssim & \|f\|_{M^{p,q}}^2.\label{md2}
\end{eqnarray}
Let $1< \frac{d}{d-\gamma}<r\leq 2, \frac{1}{r}+ \frac{1}{r'}=1.$  Note that $\frac{1}{r'}+1= \frac{1}{r_1}+ \frac{1}{r_2},$ where $r_1=r_2:= \frac{2r}{2r-1} \in [1,2]$, and $r_1' \in [2, \infty]$ where $\frac{1}{r_1}+ \frac{1}{r_1^{'}}=1.$ 
Now  using  Young's inequality for convolution, Lemma \ref{rl} \eqref{rcs}, Lemma \ref{rl} \eqref{ir}, and Lemma \ref{rl} \eqref{ic},  we obtain
\begin{eqnarray}
\|k_2\widehat{|f|^2} \|_{L^1} & \leq & \|k_2\|_{L^r}  \|\widehat{|f|^2}\|_{L^{r'}} \nonumber\\
& \lesssim &  \|\hat{f} \ast \hat{\bar{f}}\|_{L^{r'}}\nonumber\\
& \lesssim & \|\hat{f}\|_{L^{r_1}} \|\hat{\bar{f}}\|_{L^{r_1}} \nonumber \\
& \lesssim &  \|f\|^2_{M^{\min\{r_1', 2\}, r_1}}  \lesssim \|f\|^2_{M^{2, r_1}}\nonumber
\end{eqnarray}
Since  $f:[\frac{d}{d-\gamma}, \infty]\to \mathbb R, f(r) = \frac{2r}{2r-1}$ is a decreasing function, by  Lemma \ref{rl}\eqref{ir}, we have
\begin{eqnarray}
\|k_2 \widehat{|f|^2}\|_{L^1}\lesssim  \|f\|^2_{M^{2, r_1}} \lesssim \|f\|^2_{M^{p,q}}. \label{md3}
\end{eqnarray}
Combining \eqref{md1}, \eqref{md2}, and  \eqref{md3}, we obtain \eqref{d1}.
\end{proof}
\begin{Lemma}\label{cl}
Let  $0<\gamma <d,$ and $1\leq p \leq 2, 1\leq q < \frac{2d}{d+\gamma}.$ For any $ f,g \in M^{p,q}(\mathbb R^{d})$, we have
$$\| (K\ast |f|^{2})f - (K\ast |g|^{2})g\|_{M^{p,q}} \lesssim  (\|f\|_{M^{p,q}}^{2}+\|f\|_{M^{p,q}}\|g\|_{M^{p,q}}+ \|g\|_{M^{p,q}}^{2}) \|f-g\|_{M^{p,q}}.$$
\end{Lemma}
\begin{proof}
By exploiting the ideas from the proof of Lemma \ref{iml}, we obtain
\begin{eqnarray}
\|(K\ast |f|^{2})(f-g)\|_{M^{p,q}} & \lesssim & \|K\ast |f|^{2}\|_{\mathcal{F}L^{1}} \|f-g\|_{M^{p,q}}\nonumber \\
& \lesssim &  \left( \|k_{1} \widehat{|f|^{2}}\|_{L^{1}} + \|k_{2} \widehat{|f|^{2}}\|_{L^{1}}  \right) \|f-g\|_{M^{p,p}}\nonumber \\
& \lesssim &  \|f\|^{2}_{M^{p,q}} \|f-g\|_{M^{p,q}}.\label{a1}
\end{eqnarray}
Let $1\leq s < \frac{d}{d-\gamma}<t\leq 2, \frac{1}{s}+\frac{1}{s'}=1, \frac{1}{t}+\frac{1}{t'}=1.$
We note that
\begin{eqnarray}
\|(K \ast (|f|^{2}- |g|^{2}))g\|_{M^{p,q}}  & \lesssim & \|K \ast (|f|^{2}-|g|^{2})\|_{\mathcal{F}L^{1}} \|g\|_{M^{p,q}}\nonumber \\
& \lesssim & \left( \|k_1\|_{L^s}   \| \widehat{ |f|^{2}- |g|^{2}}\|_{L^{s'}} + \|k_2\|_{L^t} \| \widehat{|f|^{2}- |g|^{2}}\|_{L^{t'}} \right) \nonumber \\
&&  \|g\|_{M^{p,p}}\nonumber \\ 
& \lesssim &  \left( \| \widehat{ |f|^{2}- |g|^{2}}\|_{L^{s'}} + \| \widehat{|f|^{2}- |g|^{2}}\|_{L^{t'}} \right)  \|g\|_{M^{p,q}}\label{a2}
\end{eqnarray}
Let $1\leq r \leq 2,$ and  $\frac{1}{r}+ \frac{1}{r'}=1.$  Note that $\frac{1}{r'}+1= \frac{1}{r_1}+ \frac{1}{r_2},$ where $r_1=r_2:= \frac{2r}{2r-1} \in [1,2]$, and $r_1' \in [2, \infty]$ where $\frac{1}{r_1}+ \frac{1}{r_1^{'}}=1.$   Now  using  Young's inequality for convolution,  and  exploiting ideas performed as in the proof of Lemma \ref{iml}, we obtain 
\begin{eqnarray*}
\| |f|^{2}- |g|^{2}\|_{ \mathcal{F}L^{r'}} & \lesssim  & \|(f-g)\bar{f} \|_{\mathcal{F}L^{r'}} + \| g(\bar{f}-\bar{g}\|_{\mathcal{F}L^{r'}}\\
& = & \|\widehat{(f-g)} \ast \hat{\bar{f}}\|_{L^{r'}} + \|\hat{g} \ast  \widehat{\overline{f-g}}\|_{L^{r'}}\\
& \lesssim &  \|f-g\|_{\mathcal{F}L^{r_1}}\|\bar{f}\|_{\mathcal{F}L^{r_1}} +  \|g\|_{\mathcal{F}L^{r_1}} \|\overline{f-g}\|_{\mathcal{F}L^{r_1}}\\
& \lesssim &  (\|f\|_{M^{2, r_1}} + \|g\|_{M^{2, r_1}})  \|f-g\|_{M^{2, r_1}}\\
& \lesssim & (\|f\|_{M^{p,q}} + \|g\|_{M^{p,q}})  \|f-g\|_{M^{p,q}}.
\end{eqnarray*}
Using this,  \eqref{a2} gives
\begin{eqnarray}
\|(K \ast (|f|^{2}- |g|^{2}))g\|_{M^{p,q}} &  \lesssim & \left( \|f\|_{M^{p,q}} + \|g\|_{M^{p,q}} \right)  \|g\|_{M^{p,q}}\nonumber \\
&&\|f-g\|_{M^{p,q}}\label{al}
\end{eqnarray}
Now taking the identity
$$(K\ast |f|^{2})f- (K\ast |g|^{2})g= (K\ast |f|^{2})(f-g) + (K \ast (|f|^{2}- |g|^{2}))g $$
into our  account, \eqref{a1} and \eqref{al}  gives the desired inequality.
\end{proof}
\begin{proof}[Proof of Theorem \ref{MT}\eqref{MTF}]
By Duhamel's formula, we note that \eqref{fHTE} can be written in the equivalent form
\begin{equation}\label{df1}
u(\cdot, t)= U(t-t_0)u_{0}-i\mathcal{A}F(u)
\end{equation}
where
\begin{equation}
\  (\mathcal{A}v)(t,x)=\int_{t_0}^{t}U(t- \tau)\, v(t,x) \, d\tau.
\end{equation}
For simplicity, we assume that $t_0=0$ and prove the local existence on $[0,T]$. Similar arguments also applies to interval of the form $[-T',0]$ for proving local solutions. 

We consider now the mapping
\begin{equation}
\mathcal{J}(u)= U(t)u_{0}-i\int_{0}^{t}U(t-\tau) \, [(K\ast |u|^{2}(\tau))u(\tau)] \, d\tau.
\end{equation}
By Proposition \ref{uf},  we have 
\begin{equation}\label{a}
\left\|U(t)u_{0}\right\|_{M^{p,q}} \leq C  (1+t)^{d\left| \frac{1}{p}-\frac{1}{2} \right|} \left\|u_{0}\right\|_{M^{p,q}}
\end{equation}
for $ t \in \mathbb R$, and where $C$ is a universal constant depending only on $d.$\\
\noindent
By Minkowski's inequality for integrals, Proposition \ref{uf}, and  Lemma \ref{iml}, we obtain

\begin{eqnarray}
\label{ed}
\left\| \int_{0}^{t} U(t-\tau) [(K\ast |u|^{2}(\tau)) u(\tau)]  \, d\tau \right\|_{M^{p,q}} 
   &\leq & T C_{T} \, \|(K\ast |u|^{2}(t)) u(t)\|_{M^{p,q}} \nonumber \\
   & \leq & TC_T \|u(t)\|_{M^{p,q}}^{3}\label{e1}
\end{eqnarray}
where $C_{T}= C (1+t)^{d\left| \frac{1}{p}-\frac{1}{2} \right|}.$\\
By \eqref{a} and \eqref{e1}, we have
\begin{eqnarray}\label{tac}
\|\mathcal{J}u\|_{C([0, T], M^{p,q})} \leq  C_{T} \left(\|u_{0}\|_{M^{p,q}} + c T \|u\|_{M^{p,q}}^{3}\right),
\end{eqnarray}
for some universal constant $c.$

 For $M>0$, put  $B_{T, M}= \{u\in C([0, T], M^{p,q}(\mathbb R^{d})):\|u\|_{C([0, T], M^{p,q})}\leq M \}$,  which is the  closed ball  of radius $M$, and centered at the origin in  $C([0, T], M^{p,q}(\mathbb R^{d}))$.  
Next, we show that the mapping $\mathcal{J}$ takes $B_{T, M}$ into itself for suitable choice of  $M$ and small $T>0$. Indeed, if we let, $M= 2C_{T}\|u_{0}\|_{M^{p,q}}$ and $u\in B_{T, M},$ from \eqref{tac} we obtain 
\begin{eqnarray}
\|\mathcal{J}u\|_{C([0, T], M^{p,q})} \leq  \frac{M}{2} + cC_{T}T M^{3}.
\end{eqnarray}
We choose a  $T$  such that  $c C_{T}TM^{2} \leq 1/2,$ that is, $T \leq \tilde{T}(\|u_0\|_{M^{p,q}}, d, \gamma)$ and as a consequence  we have
\begin{eqnarray}
\|\mathcal{J}u\|_{C([0, T], M^{p,q})} \leq \frac{M}{2} + \frac{M}{2}=M,
\end{eqnarray}
that is, $\mathcal{J}u \in B_{T, M}.$
By Lemma \ref{cl}, and the arguments as before, we obtain
\begin{eqnarray}
\|\mathcal{J}u- \mathcal{J}v\|_{C([0, T], M^{p,q})} \leq \frac{1}{2} \|u-v\|_{C([0, T], M^{p,q})}.
\end{eqnarray}
Therefore, using the  Banach's contraction mapping principle, we conclude that $\mathcal{J}$ has a fixed point in $B_{T, M}$ which is a solution of \eqref{df1}. 

Now we shall see that the solution constructed before is global in time.  In fact, in view of Proposition \ref{miD}, 
to prove Theorem \ref{MT}\eqref{MTF}, it suffices to prove  that the modulation space norm of $u$, that is, $\|u\|_{M^{p,q}}$ cannot become unbounded in finite time.
In view of \eqref{dc} and to use the Hausdorff-Young inequality we let $1< \frac{d}{d-\gamma} <q \leq 2,$ 
and we obtain
\begin{eqnarray}
\|u(t)\|_{M^{p,q}} & \lesssim &  C_{T} \left( \|u_{0}\|_{M^{p,q}} + \int_{0}^{t} \|(K\ast |u(\tau)|^{2}) u(\tau)\|_{M^{p,q}}d\tau \right)\nonumber \\
& \lesssim &  C_{T} \left ( \|u_{0}\|_{M^{p,q}} + \int_{0}^{t} \|K\ast |u(\tau)|^{2}\|_{\mathcal{F}L^{1}} \|u(\tau)\|_{M^{p,q}} d\tau \right) \nonumber\\
& \lesssim & C_{T}  \|u_{0}\|_{M^{p,q}} + C_{T}\int_{0}^{t} \left( \|k_{1}\|_{L^{1}} \|u(\tau)\|_{L^{2}}^{2}+ \|k_{2}\|_{L^{q}} \|\widehat{|u(\tau)|^{2}}\|_{L^{q'}}
\right) \nonumber \\
&&\|u(\tau)\|_{M^{p,q}} d\tau \nonumber\\
& \lesssim & C_{T}\|u_{0}\|_{M^{p,q}} + C_{T} \int_{0}^{t} \left(  \|k_{1}\|_{L^{1}} \|u_{0}\|_{L^{2}}^{2}+ \|k_{2}\|_{L^{q}} \||u(\tau) |^{2}\|_{L^{q}}\right) \nonumber  \\
&& \|u(\tau)\|_{M^{p,q}} d\tau\nonumber\\
& \lesssim & C_{T}\|u_{0}\|_{M^{p,q}}+ C_{T} \int_{0}^{t}\|u(\tau)\|_{M^{p,q}} d\tau  \nonumber \\
&& + C_{T} \int_{0}^{t} \|u(\tau)\|_{L^{2q}}^{2} \|u(\tau)\|_{M^{p,q} } d\tau,\nonumber
 \end{eqnarray}
 where we have used  \eqref{mp},  H\"older's inequality, and  the conservation of the $L^{2}-$norm of $u$. \\
We note that the requirement on $q$ can be fulfilled if and only if $0<\gamma <d/2.$ To apply Proposition \ref{mi}, we let $\beta>1$ and $(2\beta, 2 q)$ is  $\alpha-$fractional admissible, that is, $\frac{\alpha}{2\beta}= d \left(\frac{1}{2}- \frac{1}{2q} \right)$ such that $\frac{1}{\beta}= \frac{d}{\alpha} \left( 1 - \frac{1}{q} \right)<1.$ This is possible provided $\frac{q-1}{q} < \frac{\alpha}{d}:$ this condition is compatible with the requirement $q> \frac{d}{d-\gamma}$ if and only if $\gamma < \alpha.$
 Using the H\"older's inequality for the last integral, we obtain
\begin{eqnarray*}
\|u(t)\|_{M^{p,q}} &  \lesssim  & C_{T}\|u_0\|_{M^{p,q}} + C_{T}\int_{0}^{t} \|u(\tau)\|_{M^{p,q}} d\tau \\
&& + C_{T}\|u\|_{L^{2\beta}([0, T], L^{2q})}^{2}\|u\|_{L^{\beta'}[0, T], M^{p,q})},
\end{eqnarray*}
where $\beta'$ is the H\"older conjugate exponent of $\beta.$
Put,
$$h(t):=\sup_{0 \leq \tau \leq t} \|u(\tau)\|_{M^{p,q}}.$$
For a given $T>0,$ $h$ satisfies an estimate of the form,
$$h(t)\lesssim C_{T}\|u_{0}\|_{M^{p,q}}+ C_{T}\int_{0}^{t} h(\tau) d\tau + C_{T} C_0(T) \left( \int_{0}^{t}h(\tau)^{\beta'} d\tau \right)^{\frac{1}{\beta'}},$$
provided that $0 \leq t \leq T,$ and where we have used the fact that $\beta'$ is finite.
Using the H\"older's inequality we infer that,
$$h(t)\lesssim  C_{T} \|u_{0}\|_{M^{p,q}} + C_{1}(T) \left(\int_{0}^{t} h(\tau)^{\beta'}d \tau \right)^{\frac{1}{\beta'}}.$$
Raising the above estimate to the power $\alpha'$, we find that
$$h(t)^{\beta'} \lesssim  C_{2}(T) \left( 1+\int_{0}^{t} h(\tau)^{\beta'} d\tau\right).$$ 
In view of  Gronwall inequality, one may conclude that  $h\in L^{\infty}([0, T]).$ Since $T>0$ is arbitrary, $h\in L^{\infty}_{loc}(\mathbb R),$ and  the proof of Theorem \ref{MT}\eqref{MTF}  follows.
\end{proof}
\begin{proof}[Proof of Theorem \ref{MT}\eqref{MTS}]  Taking Lemmas \ref{iml} and \ref{cl}, and Proposition \ref{carlep}  into account, and exploiting ideas from Theorem \ref{MT}\eqref{MTF},  we can produce the proof.
\end{proof}
\subsection{Global well-posednes in  $\mathcal{F}L^1(\mathbb R^d) \cap L^2(\mathbb R^d)$} In this subsection, we shall prove Theorem \ref{jcm}.
The following lemma is easy to observe:
\begin{Lemma} \label{of} Let $f,g \in \mathcal{F}L^1(\mathbb R^d)\cap  L^2(\mathbb R^d).$
\noindent
\begin{enumerate}
\item \label{of1} The space $\mathcal{F}L^1(\mathbb R^d)$ is an algebra under point wise multiplication, with  norm inequality
$$\|fg\|_{\mathcal{F}L^1} \leq \|f\|_{\mathcal{F}L^1} \|g\|_{\mathcal{F}L^1}.$$
\item \label{of2} For all $t\in \mathbb R,$  the fractional Schr\"odinger propagator $e^{it(-\Delta)^{\alpha/2}}$  is unitary on $\mathcal{F}L^1(\mathbb R^d).$
\item \label{of3}  Let  $K$ be given by \eqref{hk} with $\lambda \in \mathbb R,$ and $ 0<\gamma < d$. Then 
$$\| (K\ast |f|^{2})f - (K\ast |g|^{2})g\|_{L^2\cap \mathcal{F}L^1} \lesssim  (\|f\|_{L^2\cap \mathcal{F}L^1}^{2}+ \|g\|_{L^2 \cap \mathcal{F}L^1}^{2}) \|f-g\|_{L^2\cap \mathcal{F}L^1}.$$
\end{enumerate}
\end{Lemma}
\begin{proof}  The proof of statement \eqref{of1} follows form the  Young's inequality for the convolution. Next we note that $\|e^{it (-\Delta)}f\|_{\mathcal{F}L^1}= \|e^{i\pi t |\xi|^{\alpha}} \hat{f}\|_{L^1}=\|f\|_{\mathcal{F}L^1}.$ This completes the proof of statement  \eqref{of2}.  For the proof of statement  \eqref{of3}, see \cite[Lemma 3.1]{carle}.
\end{proof}
Using Lemma \ref{of}, and standard fixed point arguments,  following proposition is easy to prove, and the proof is left to the reader.
\begin{Proposition}\label{lfo} Let  $K$ be given by \eqref{hk} with $\lambda \in \mathbb R,$ and $ 0<\gamma < d$. If $u_0\in L_{rad}^2(\mathbb R^d)\cap \mathcal{F}L^1(\mathbb R^d),$ then there exists $T>0$ depending only on $\lambda, \gamma, d$ and $\|u_0\|_{L^2 \cap \mathcal{F}L^1}$ 
and a  unique solution of \eqref{fHTE}  such that $u \in C([0,T], L^2\cap \mathcal{F}L^1)$.
\end{Proposition}
\begin{proof}[Proof of Theorem \ref{jcm}]  
Taking Proposition \ref{lfo}  and  Proposition \ref{miD} into account, 
to prove Theorem \ref{jcm}, it suffices to prove  that the Fourrier algebra norm of $u$, that is, $\|u\|_{\mathcal{F}L^1}$ cannot become unbounded in finite time.
In view of \eqref{dc} and to use the Hausdorff-Young inequality we let $1< \frac{d}{d-\gamma} <q \leq 2,$ 
and we obtain
\begin{eqnarray}
\|u(t)\|_{\mathcal{F}L^1} & \lesssim &   \|u_{0}\|_{\mathcal{F}L^1} + \int_{0}^{t} \|(K\ast |u(\tau)|^{2}) u(\tau)\|_{\mathcal{F}L^1}d\tau \nonumber \\
& \lesssim &    \|u_{0}\|_ {\mathcal{F}L^1} + \int_{0}^{t} \|K\ast |u(\tau)|^{2}\|_{\mathcal{F}L^{1}} \|u(\tau)\|_{\mathcal{F}L^1} d\tau  \nonumber\\
& \lesssim &   \|u_{0}\|_{\mathcal{F}L^1} + \int_{0}^{t} \left( \|k_{1}\|_{L^{1}} \|u(\tau)\|_{L^{2}}^{2}+ \|k_{2}\|_{L^{q}} \|\widehat{|u(\tau)|^{2}}\|_{L^{q'}}
\right) \nonumber \\
&&\|u(\tau)\|_{\mathcal{F}L^1} d\tau \nonumber\\
& \lesssim & \|u_{0}\|_{\mathcal{F}L^1} +  \int_{0}^{t} \left(  \|k_{1}\|_{L^{1}} \|u_{0}\|_{L^{2}}^{2}+ \|k_{2}\|_{L^{q}} \||u(\tau) |^{2}\|_{L^{q}}\right) \nonumber  \\
&& \|u(\tau)\|_{\mathcal{F}L^1} d\tau\nonumber\\
& \lesssim & \|u_{0}\|_{\mathcal{F}L^1}+  \int_{0}^{t}\|u(\tau)\|_{\mathcal{F}L^1} d\tau  +  \int_{0}^{t} \|u(\tau)\|_{L^{2q}}^{2} \|u(\tau)\|_{\mathcal{F}L^1} d\tau,\nonumber
 \end{eqnarray}
 where we have used  Lemma \ref{of},  H\"older's inequality, and  the conservation of the $L^{2}-$norm of $u$. \\
We note that the requirement on $q$ can be fulfilled if and only if $0<\gamma <d/2.$ To apply Proposition \ref{mi}, we let $\beta>1$ and $(2\beta, 2 q)$ is  $\alpha-$fractional admissible, that is, $\frac{\alpha}{2\beta}= d \left(\frac{1}{2}- \frac{1}{2q} \right)$ such that $\frac{1}{\beta}= \frac{d}{\alpha} \left( 1 - \frac{1}{q} \right)<1.$ This is possible provided $\frac{q-1}{q} < \frac{\alpha}{d}:$ this condition is compatible with the requirement $q> \frac{d}{d-\gamma}$ if and only if $\gamma < \alpha.$
 Using the H\"older's inequality for the last integral, we obtain
\begin{eqnarray*}
\|u(t)\|_{\mathcal{F}L^1} &  \lesssim  & \|u_0\|_{\mathcal{F}L^1} + \int_{0}^{t} \|u(\tau)\|_{\mathcal{F}L^1} d\tau \\
&& + C_{T}\|u\|_{L^{2\beta}([0, T], L^{2q})}^{2}\|u\|_{L^{\beta'}[0, T],  \mathcal{F}L^1)}
\end{eqnarray*}
where $\beta'$ is the H\"older conjugate exponent of $\beta.$
Put,
$$h(t):=\sup_{0 \leq \tau \leq t} \|u(\tau)\|_{\mathcal{F}L^1}.$$
For a given $T>0,$ $h$ satisfies an estimate of the form,
$$h(t)\lesssim \|u_{0}\|_{\mathcal{F}L^1}+ \int_{0}^{t} h(\tau) d\tau + C_0(T) \left( \int_{0}^{t}h(\tau)^{\beta'} d\tau \right)^{\frac{1}{\beta'}},$$
provided that $0 \leq t \leq T,$ and where we have used the fact that $\beta'$ is finite.
Using the H\"older's inequality we infer that,
$$h(t)\lesssim   \|u_{0}\|_{\mathcal{F}L^1} + C_{1}(T) \left(\int_{0}^{t} h(\tau)^{\beta'}d \tau \right)^{\frac{1}{\beta'}}.$$
Raising the above estimate to the power $\beta'$, we find that
$$h(t)^{\beta'} \lesssim  C_{2}(T) \left( 1+\int_{0}^{t} h(\tau)^{\beta'} d\tau\right).$$ 
In view of  Gronwall inequality, one may conclude that  $h\in L^{\infty}([0, T]).$ Since $T>0$ is arbitrary, $h\in L^{\infty}_{loc}(\mathbb R),$ and  the proof of Theorem \ref{jcm}  follows.
\end{proof}
\section{Concluding Remarks}
\begin{enumerate}
\item We have proved Theorem \ref{MT}  with range $1\leq p \leq 2, 1\leq q < \frac{2d}{d+\gamma},$ it would be interesting to know whether the range of $q$ can be improved in Theorem \ref{MT}.

\item  In view of Proposition \ref{uf}, and the fact that $\mathcal{F}L^1(\mathbb R^d) \hookrightarrow M^{\infty, 1}(\mathbb R^d),$ it would be interesting  to know whether the analogue of Theorem \ref{jcm} is true or not for the Cauchy data in $M^{\infty,1}(\mathbb R^d)\cap L_{rad}^2(\mathbb R^d).$
\end{enumerate}

\noindent
{\textbf{Acknowledgment}:} 
The author  is  thankful to   IUSTF and Indo-US SERB and DST-INSPIRE, and TIFR CAM for the  support.  The author is very grateful to Professor Kasso Okoudjou for his suggestions,  hospitality and   arranging  excellent research facilities at   the University of Maryland.

\end{document}